\documentclass[english]{article}
\usepackage[T1]{fontenc}
\usepackage[latin9]{inputenc}
\usepackage{textcomp}
\usepackage{mathrsfs}
\usepackage{amsmath}
\usepackage{amsthm}
\usepackage{amssymb}
\usepackage{wasysym}

\makeatletter
\numberwithin{equation}{section}
\numberwithin{figure}{section}
\theoremstyle{plain}
\newtheorem{thm}{\protect\theoremname}[section]
  \theoremstyle{plain}
  \newtheorem{prop}[thm]{\protect\propositionname}
  \theoremstyle{plain}
  \newtheorem{lem}[thm]{\protect\lemmaname}
  \theoremstyle{remark}
  \newtheorem{rem}[thm]{\protect\remarkname}
  \theoremstyle{definition}
  \newtheorem{defn}[thm]{\protect\definitionname}
  \theoremstyle{plain}
  \newtheorem{cor}[thm]{\protect\corollaryname}
  \theoremstyle{plain}
  \newtheorem{conjecture}[thm]{\protect\conjecturename}

\makeatother

\usepackage{babel}
  \providecommand{\conjecturename}{Conjecture}
  \providecommand{\corollaryname}{Corollary}
  \providecommand{\definitionname}{Definition}
  \providecommand{\lemmaname}{Lemma}
  \providecommand{\propositionname}{Proposition}
  \providecommand{\remarkname}{Remark}
\providecommand{\theoremname}{Theorem}

\begin{document}

\title{The Congruence Subgroup Problem\\
for low rank Free and Free Metabelian groups}

\author{David El-Chai Ben-Ezra, Alexander Lubotzky}

\maketitle
\[
To\,\,Efim\,\,Zelmanov,\,\,a\,\,friend\,\,and\,\,a\,\,leader
\]

\begin{abstract}
The congruence subgroup problem for a finitely generated group $\Gamma$
asks whether $\widehat{Aut\left(\Gamma\right)}\to Aut(\hat{\Gamma})$
is injective, or more generally, what is its kernel $C\left(\Gamma\right)$?
Here $\hat{X}$ denotes the profinite completion of $X$. 

In this paper we first give two new short proofs of two known results
(for $\Gamma=F_{2}$ and $\Phi_{2}$) and a new result for $\Gamma=\Phi_{3}$:

~~~~~$\left(1\right)$ $C\left(F_{2}\right)=\left\{ e\right\} $
when $F_{2}$ is the free group on two generators.

~~~~~$\left(2\right)$ $C\left(\Phi_{2}\right)=\hat{F}_{\omega}$
when $\Phi_{n}$ is the free metabelian group on $n$ gene-

~~~~~~~~~~rators, and $\hat{F}_{\omega}$ is the free profinite
group on $\aleph_{0}$ generators.

~~~~~$\left(3\right)$ $C\left(\Phi_{3}\right)$ contains $\hat{F}_{\omega}$.

Results $\left(2\right)$ and $\left(3\right)$ should be contrasted
with an upcoming result of the first author showing that $C\left(\Phi_{n}\right)$
is abelian for $n\geq4$.\\

\end{abstract}
\textbf{Mathematics Subject Classification (2010):} primary: 19B37,
20H05, secondary: 20E18, 20E36, 20E05.\textbf{}\\
\textbf{}\\
\textbf{Key words and phrases:} congruence subgroup problem, profinite
groups, automorphism groups, free groups, free metabelian groups.

\section{Introduction}

The classical congruence subgroup problem (CSP) asks for, say, $G=SL_{n}\left(\mathbb{Z}\right)$
or $G=GL_{n}\left(\mathbb{Z}\right)$, whether every finite index
subgroup of $G$ contains a principal congruence subgroup, i.e. a
subgroup of the form $G\left(m\right)=\ker\left(G\to GL_{n}\left(\mathbb{Z}/m\mathbb{Z}\right)\right)$
for some $0\neq m\in\mathbb{Z}$. Equivalently, it asks whether the
natural map $\hat{G}\to GL_{n}(\hat{\mathbb{Z}})$ is injective, where
$\hat{G}$ and $\hat{\mathbb{Z}}$ are the profinite completions of
the group $G$ and the ring $\mathbb{Z}$, respectively. More generally,
the CSP asks what is the kernel of this map. It is a classical $19^{\underline{th}}$
century result that the answer is negative for $n=2$. Moreover (but
not so classical, cf. \cite{key-17}, \cite{key-4}), the kernel,
in this case, is $\hat{F}_{\omega}$ - the free profinite group on
a countable number of generators. On the other hand, for $n\geq3$,
the map is injective and the kernel is therefore trivial.

The CSP can be generalized as follows: Let $\Gamma$ be a group and
$M$ a finite index characteristic subgroup of it. Denote: 
\[
G\left(M\right)=\ker\left(Aut\left(\Gamma\right)\to Aut\left(\Gamma/M\right)\right).
\]
Such a finite index normal subgroup of $G=Aut\left(\Gamma\right)$
will be called a ``principal congruence subgroup'' and a finite
index subgroup of $G$ which contains such a $G\left(M\right)$ for
some $M$ will be called a ``congruence subgroup''. Now, the CSP
for $\Gamma$ asks whether every finite index subgroup of $G$ is
a congruence subgroup. When $\Gamma$ is finitely generated, $Aut(\hat{\Gamma})$
is profinite and the CSP is equivalent to the question (cf. \cite{key-5},
$\varoint$1 and $\varoint$3): Is the map $\hat{G}=\widehat{Aut\left(\Gamma\right)}\to Aut(\hat{\Gamma})$
injective? More generally, it asks what is the kernel $C\left(\Gamma\right)$
of this map. 

As $GL_{n}\left(\mathbb{Z}\right)=Aut\left(\mathbb{Z}^{n}\right)$,
the classical congruence subgroup results mentioned above can therefore
be reformulated as $C\left(A_{2}\right)=\hat{F}_{\omega}$ while $C\left(A_{n}\right)=\left\{ e\right\} $
for $n\geq3$, when $A_{n}=\mathbb{Z}^{n}$ is the free abelian group
on $n$ generators.

Very few results are known when $\Gamma$ is non-abelian. A very surprising
result was proved in \cite{key-3} by Asada by methods of algebraic
geometry:
\begin{thm}
\label{thm:Free 2}$C\left(F_{2}\right)=\left\{ e\right\} $, i.e.,
the free group on two generators has the congruence subgroup property,
namely $\widehat{Aut\left(F_{2}\right)}\to Aut(\hat{F}_{2})$ is injective.
\end{thm}
A purely group theoretic proof for this theorem was given by Bux-Ershov-Rapinchuk
\cite{key-5}. Our first goal in this paper is to give an easier and
more direct proof of Theorem \ref{thm:Free 2}, which also give a
better quantitative estimate: we give an explicitly constructed congruence
subgroup $G\left(M\right)$ of $Aut\left(F_{2}\right)$ which is contained
in a given finite index subgroup $H$ of $Aut\left(F_{2}\right)$
of index $n$. Our estimates on the index of $M$ in $F_{2}$ as a
function of $n$ are substantially better than those of \cite{key-5}
- see Theorems \ref{thm:explicit-BER} and \ref{thm:explicit}.

We then turn to $\Gamma=\Phi_{2}$, the free metabelian group on two
generators. The initial treatment of $\Phi_{2}$ is similar to $F_{2}$,
but quite surprisingly, the first named author showed in \cite{key-6}
a negative answer, i.e. $C\left(\Phi_{2}\right)\neq\left\{ e\right\} $.
We also give a shorter proof of this result, deducing that:
\begin{thm}
\label{thm:.metabelian 2}$C\left(\Phi_{2}\right)=\hat{F}_{\omega}$.
\end{thm}
We then go ahead from 2 to 3 and prove:
\begin{thm}
\label{thm: metabelian 3}$C\left(\Phi_{3}\right)$ contains a copy
of $\hat{F}_{\omega}$. In particular, the congruence subgroup property
(strongly) fails for $\Phi_{3}$.
\end{thm}
This is also surprising, especially if compared with an upcoming paper
of the first author \cite{key-7} showing that $C\left(\Phi_{n}\right)$
is abelian for $n\geq4$. So, while the dichotomy for the abelian
case $A_{n}=\mathbb{Z}^{n}$ is between $n=2$ and $n\geq3$, for
the metabelian case, it is between $n=2,3$ and $n\geq4$.

A main ingredient of the proof of Theorem \ref{thm: metabelian 3}
is showing that $Aut\left(\Phi_{3}\right)$ is large, i.e. it has
a finite index subgroup which is mapped onto a non-abelian free group.
For this we use the method developed by Grunewald and the second author
in \cite{key-13-1} to produce arithmetic quotients of $Aut\left(F_{n}\right)$.
In particular, it is shown there that $Aut\left(F_{3}\right)$ is
large. Our starting point to prove Theorem \ref{thm: metabelian 3}
is the observation that the same proof shows also that $Aut\left(\Phi_{3}\right)$
is large.

In our proof of Theorem \ref{thm:.metabelian 2}, the largeness of
$Aut\left(\Phi_{2}\right)$ is also playing a crucial role. But, a
word of warning is needed here: largeness of $Aut\left(\Gamma\right)$
by itself is not sufficient to deduce negative answer for the CSP
for $\Gamma$. For example, $Aut\left(F_{2}\right)$ is large but
has an affirmative answer for the CSP. At the same time, as mentioned
above, $Aut\left(F_{3}\right)$ is large and we do not know whether
$F_{3}$ has the congruence subgroup property or not. To prove Theorem
\ref{thm: metabelian 3} we use the largeness of $Aut\left(\Phi_{3}\right)$
combined with the fact that every non-abelian finite simple group
which is involved in $Aut(\hat{\Phi}_{3})$ is already involved in
$GL_{3}\left(R\right)$ for some finite commutative ring $R$, as
we will show below.

The paper is organized as follows: In $\varoint$\ref{sec:SCP F2}
we give a short proof for Theorem \ref{thm:Free 2} and in $\varoint$\ref{sec:CSP meta 2}
for Theorem \ref{thm:.metabelian 2}. The $4^{\underline{th}}$ section
is devoted to the proof of Theorem \ref{thm: metabelian 3}. We close
in $\varoint$\ref{sec:Remarks} with some remarks and open problems,
about free nilpotent and solvable groups.

\textbf{Acknowledgements:} The first author wishes to offer his deepest
thanks to the Rudin foundation trustees for their generous support
during the period of the research. This paper is part of his PhD's
thesis at the Hebrew University of Jerusalem. The second author wishes
to acknowledge support by the ISF, the NSF, the ERC, Dr. Max Rossler,
the Walter Haefner Foundation and the ETH Foundation.

\section{The CSP for $F_{2}$\label{sec:SCP F2}}

Before we start, let us quote some general propositions which Bux-Ershov-Rapinchuk
bring throughout their paper.
\begin{prop}
\textup{\label{prop:BER-exact}(cf. \cite{key-5}, Lemma 2.1)} Let:
\[
1\to G_{1}\overset{\alpha}{\rightarrow}G_{2}\overset{\beta}{\rightarrow}G_{3}\to1
\]
be an exact sequence of groups. Assume that $G_{1}$ is finitely generated
and that the center of its profinite completion $\hat{G}_{1}$ is
trivial. Then, the sequence of the profinite completions 
\[
1\to\hat{G}_{1}\overset{\hat{\alpha}}{\rightarrow}\hat{G}_{2}\overset{\hat{\beta}}{\rightarrow}\hat{G}_{3}\to1
\]
is also exact.
\end{prop}

\begin{prop}
\textup{\label{prop:BER-F2}(cf. \cite{key-5}, Corollaries 2.3, 2.4.
and 2.7) }Let $F$ be the free group on the set $X$, $\left|X\right|\geq2$.
then:

1. The center of $\hat{F}$, the profinite completion of $F$, is
trivial.

2. If $x,y\in X$, $x\neq y$, then the centralizer of $[y,x]$ in
$\hat{F}$ is $Z_{\hat{F}}\left([y,x]\right)=\overline{\left\langle [y,x]\right\rangle },$
the closure of the cyclic group generated by $[y,x]$. 
\end{prop}
We start now with the following lemma whose easy proof is left to
the reader:
\begin{lem}
\label{lem:cong-sub}Let $H\leq G=Aut\left(\Gamma\right)$ be a congruence
subgroup. Then: 
\[
\ker(\hat{G}\to Aut(\hat{\Gamma}))=\ker(\hat{H}\to Aut(\hat{\Gamma}))
\]
In particular, the map $\hat{G}\to Aut(\hat{\Gamma})$ is injective
if and only if the map $\hat{H}\to Aut(\hat{\Gamma})$ is injective.
\end{lem}
Denote now $F_{2}=\left\langle x,y\right\rangle $ = the free group
on $x$ and $y$. It is a well known theorem of Nielsen (cf. \cite{key-1},
3.5) that the kernel of the natural surjective map: 
\[
Aut\left(F_{2}\right)\to Aut\left(F_{2}/F'_{2}\right)=Aut\left(\mathbb{Z}^{2}\right)=GL_{2}\left(\mathbb{Z}\right)
\]
is $Inn\left(F_{2}\right)$, the inner automorphism group of $F_{2}$.
It is also well known that the group $\left\langle \left(\begin{array}{cc}
1 & 2\\
0 & 1
\end{array}\right),\left(\begin{array}{cc}
1 & 0\\
2 & 1
\end{array}\right)\right\rangle \cong F_{2}$ is free on two generators and of finite index in $GL_{2}\left(\mathbb{Z}\right)$
which contains $\ker\left(GL_{2}\left(\mathbb{Z}\right)\to GL_{2}\left(\mathbb{Z}/4\mathbb{Z}\right)\right)$.
Now, if we denote the preimage of $\left\langle \left(\begin{array}{cc}
1 & 2\\
0 & 1
\end{array}\right),\left(\begin{array}{cc}
1 & 0\\
2 & 1
\end{array}\right)\right\rangle $ under the map $Aut\left(F_{2}\right)\to GL_{2}\left(\mathbb{Z}\right)$
by $Aut'\left(F_{2}\right)$, then $Aut'\left(F_{2}\right)$ is of
finite index in $Aut\left(F_{2}\right)$ and contains the principal
congruence subgroup:
\[
\ker\left(Aut\left(F_{2}\right)\to GL_{2}\left(\mathbb{Z}\right)\to GL_{2}\left(\mathbb{Z}/4\mathbb{Z}\right)=Aut\left(F_{2}/\left(F_{2}^{4}F'_{2}\right)\right)\right).
\]
So, by Lemma \ref{lem:cong-sub} it is enough to prove that $\widehat{Aut'\left(F_{2}\right)}\to Aut(\hat{F}_{2})$
is injective. 

Now, by the description above, we deduce the exact sequence:
\[
1\to Inn\left(F_{2}\right)\to Aut'\left(F_{2}\right)\to\left\langle \left(\begin{array}{cc}
1 & 2\\
0 & 1
\end{array}\right),\left(\begin{array}{cc}
1 & 0\\
2 & 1
\end{array}\right)\right\rangle \to1.
\]
As $\left\langle \left(\begin{array}{cc}
1 & 2\\
0 & 1
\end{array}\right),\left(\begin{array}{cc}
1 & 0\\
2 & 1
\end{array}\right)\right\rangle $ is free , this sequence splits by the map:
\[
\left(\begin{array}{cc}
1 & 2\\
0 & 1
\end{array}\right)\mapsto\alpha=\begin{cases}
x\mapsto & x\\
y\mapsto & yx^{2}
\end{cases}\;,\qquad\left(\begin{array}{cc}
1 & 0\\
2 & 1
\end{array}\right)\mapsto\beta=\begin{cases}
x\mapsto & xy^{2}\\
y\mapsto & y
\end{cases}
\]
and thus: $Aut'\left(F_{2}\right)=Inn\left(F_{2}\right)\rtimes\left\langle \alpha,\beta\right\rangle $.
By Propositions \ref{prop:BER-exact} and \ref{prop:BER-F2}, the
exact sequence: $1\to Inn\left(F_{2}\right)\to Aut'\left(F_{2}\right)\to\left\langle \alpha,\beta\right\rangle \to1$
yields the exact sequence: 
\[
1\to\widehat{Inn\left(F_{2}\right)}\to\widehat{Aut'\left(F_{2}\right)}\to\widehat{\left\langle \alpha,\beta\right\rangle }\to1
\]
which gives: 
\[
\widehat{Aut'\left(F_{2}\right)}=\widehat{Inn\left(F_{2}\right)}\rtimes\widehat{\left\langle \alpha,\beta\right\rangle }
\]
Thus, all we need to show is that the following map is injective:
\[
\widehat{Inn\left(F_{2}\right)}\rtimes\widehat{\left\langle \alpha,\beta\right\rangle }\to Aut(\hat{F}_{2}).
\]

We will prove this, in three parts: The first part is that the map
$\widehat{Inn\left(F_{2}\right)}\to Aut(\hat{F}_{2})$ is injective,
but this is obvious as $\widehat{Inn\left(F_{2}\right)}\cong\hat{F}_{2}$
is mapped isomorphically to $Inn(\hat{F}_{2})\cong\hat{F}_{2}$. The
second part is to show that the map $\rho:\widehat{\left\langle \alpha,\beta\right\rangle }\to Aut(\hat{F}_{2})$
is injective, and the last part is to show that the intersection of
the images of $\widehat{Inn\left(F_{2}\right)}$ and $\widehat{\left\langle \alpha,\beta\right\rangle }$
in $Aut(\hat{F}_{2})$ is trivial, i.e. $Inn(\hat{F}_{2})\cap\textrm{Im}\rho=\left\{ e\right\} $.
\\

So it remains to prove the next two lemmas, Lemma \ref{lem:main-lemma}
and Lemma \ref{lem: part 3}:
\begin{lem}
\label{lem:main-lemma}The map $\widehat{\left\langle \alpha,\beta\right\rangle }\to Aut(\hat{F}_{2})$
is injective.
\end{lem}
Before proving the lemma, we recall a classical result of Schreier:
\begin{thm}
\label{thm:Schreier}\textup{(cf. \cite{key-1}, 2.3 and 2.4)} Let
$F$ be the free group on the set $X$ where $\left|X\right|=n$,
and $\Delta$ a subgroup of $F$ of index $m$. Let $T$ be a right
Schreier transversal of $\Delta$ (i.e. a system of representatives
of right cosets containing the identity, such that the initial segment
of any element of $T$ is also in $T$). Then:

1. $\Delta$ is a free group on $m\cdot\left(n-1\right)+1$ elements.

2. The set $\left\{ tx\left(\overline{tx}\right)^{-1}\neq e\,|\,t\in T,\,x\in X\right\} $
is a free generating set for $\Delta$, where for every $g\in F$
we denote by $\bar{g}$ the unique element in $T$ satisfying $\Delta g=\Delta\bar{g}$.\end{thm}
\begin{proof}
(of Lemma \ref{lem:main-lemma}) Define $\Delta=\ker(F_{2}\to\left(\mathbb{Z}/2\mathbb{Z}\right)^{2})$.
This is a characteristic subgroup of index $4$ in $F_{2}$, that
by the first part of Theorem \ref{thm:Schreier}, is isomorphic to
$F_{5}$. We also have: $\hat{\Delta}=\ker(\hat{F}_{2}\to\left(\mathbb{Z}/2\mathbb{Z}\right)^{2})$,
and therefore, there is a natural homomorphism: $Aut(\hat{F}_{2})\to Aut(\hat{\Delta})\cong Aut(\hat{F}_{5})$
which induces the composition $\widehat{\left\langle \alpha,\beta\right\rangle }\to Aut(\hat{F}_{2})\to Aut(\hat{\Delta})$.
Thus, it is enough to show that the composition map $\widehat{\left\langle \alpha,\beta\right\rangle }\to Aut(\hat{\Delta})$
is injective.

Now, let $X=\left\{ x,y\right\} $ and $T=\left\{ 1,x,y,xy\right\} $
be a right Schreier transversal of $\Delta$. By applying the second
part of Theorem \ref{thm:Schreier} for $X$ and $T$, we get the
following set of free generators for $\Delta$:
\[
e_{1}=x^{2},\quad e_{2}=yxy^{-1}x^{-1},\quad e_{3}=y^{2},\quad e_{4}=xyxy^{-1},\quad e_{5}=xy^{2}x^{-1}
\]

Hence, the automorphisms $\alpha$ and $\beta$ act on $\Delta$ in
the following way:
\begin{eqnarray*}
\alpha & = & \begin{cases}
e_{1}=x^{2}\mapsto x^{2} & =e_{1}\\
e_{2}=yxy^{-1}x^{-1}\mapsto yxy^{-1}x^{-1} & =e_{2}\\
e_{3}=y^{2}\mapsto yx^{2}yx^{2} & =e_{2}e_{4}e_{3}e_{1}\\
e_{4}=xyxy^{-1}\mapsto xyxy^{-1} & =e_{4}\\
e_{5}=xy^{2}x^{-1}\mapsto xyx^{2}yx & =e_{4}e_{2}e_{5}e_{1}
\end{cases}\\
\\
\beta & = & \begin{cases}
e_{1}=x^{2}\mapsto xy^{2}xy^{2} & =e_{5}e_{1}e_{3}\\
e_{2}=yxy^{-1}x^{-1}\mapsto yxy^{-1}x^{-1} & =e_{2}\\
e_{3}=y^{2}\mapsto y^{2} & =e_{3}\\
e_{4}=xyxy^{-1}\mapsto xy^{3}xy & =e_{5}e_{4}e_{3}\\
e_{5}=xy^{2}x^{-1}\mapsto xy^{2}x^{-1} & =e_{5}
\end{cases}
\end{eqnarray*}

Let us now define the map $\pi:\Delta\to\left\langle \alpha,\beta\right\rangle \cong F_{2}$
(yes! these are the same $\alpha$ and $\beta$) by the following
way:
\[
\pi=\begin{cases}
e_{1}\mapsto & \alpha\\
e_{2}\mapsto & 1\\
e_{3}\mapsto & \beta\\
e_{4}\mapsto & \alpha^{-1}\\
e_{5}\mapsto & \beta^{-1}
\end{cases}
\]

It is easy to see that $N=\ker\pi$ is the normal subgroup of $\Delta$
generated as a normal subgroup by $e_{2},\,e_{1}e_{4}$ and $e_{3}e_{5}$,
and that $N$ is invariant under the action of the automorphisms $\alpha$
and $\beta$, since:
\[
\begin{cases}
\alpha\left(e_{2}\right) & =e_{2}\in N\\
\alpha\left(e_{1}e_{4}\right) & =e_{1}e_{4}\in N\\
\alpha\left(e_{3}e_{5}\right) & =e_{2}e_{4}e_{3}e_{1}e_{4}e_{2}e_{5}e_{1}\\
 & =e_{4}\left(\left(e_{4}^{-1}e_{2}e_{4}\right)\left(e_{3}\left(\left(e_{1}e_{4}\right)e_{2}\right)e_{3}^{-1}\right)\left(e_{3}e_{5}\right)\left(e_{1}e_{4}\right)\right)e_{4}^{-1}\in N
\end{cases}
\]
\[
\begin{cases}
\beta\left(e_{2}\right) & =e_{2}\in N\\
\beta\left(e_{1}e_{4}\right) & =e_{5}e_{1}e_{3}e_{5}e_{4}e_{3}=e_{5}\left(\left(e_{1}\left(e_{3}e_{5}\right)e_{1}^{-1}\right)\left(e_{1}e_{4}\right)\left(e_{3}e_{5}\right)\right)e_{5}^{-1}\in N\\
\beta\left(e_{3}e_{5}\right) & =e_{3}e_{5}\in N
\end{cases}
\]

Therefore, the homomorphism $\widehat{\left\langle \alpha,\beta\right\rangle }\to Aut(\hat{\Delta})$
induces a homomorphism: $\widehat{\left\langle \alpha,\beta\right\rangle }\to Aut\left(\widehat{\left\langle \alpha,\beta\right\rangle }\right)$,
and thus it is enough to show that the last map is injective. Now,
under this map, $\alpha$ and $\beta$ act on $\left\langle \alpha,\beta\right\rangle $
in the following way:
\begin{eqnarray*}
\alpha & = & \begin{cases}
\alpha=e_{1}N\mapsto\alpha\left(e_{1}N\right)=\alpha\left(e_{1}\right)N=e_{1}N & =\alpha\\
\beta=e_{3}N\mapsto\alpha\left(e_{3}N\right)=\alpha\left(e_{3}\right)N=e_{2}e_{4}e_{3}e_{1}N & =\alpha^{-1}\beta\alpha
\end{cases}\\
\beta & = & \begin{cases}
\alpha=e_{1}N\mapsto\beta\left(e_{1}N\right)=\beta\left(e_{1}\right)N=e_{5}e_{1}e_{3}N & \,\,\,\,\,\,=\beta^{-1}\alpha\beta\\
\beta=e_{3}N\mapsto\beta\left(e_{3}N\right)=\beta\left(e_{3}\right)N=e_{3}N & \,\,\,\,\,\,=\beta
\end{cases}
\end{eqnarray*}

Namely, $\alpha$ and $\beta$ act via $\pi$ on $\widehat{\left\langle \alpha,\beta\right\rangle }$
by the inner automorphisms $\alpha$ and $\beta$ and hence $\widehat{\left\langle \alpha,\beta\right\rangle }$
is mapped isomorphically to $Inn\left(\widehat{\left\langle \alpha,\beta\right\rangle }\right)$,
yielding that the map $\widehat{\left\langle \alpha,\beta\right\rangle }\to Aut\left(\widehat{\left\langle \alpha,\beta\right\rangle }\right)$
is injective and $\widehat{\left\langle \alpha,\beta\right\rangle }\to Aut(\hat{F}_{2})$
is injective as well, as required.\end{proof}
\begin{lem}
\label{lem: part 3}$Inn(\hat{F}_{2})\cap\textrm{Im}\rho=\left\{ e\right\} $,
where $\rho:\widehat{\left\langle \alpha,\beta\right\rangle }\to Aut(\hat{F}_{2})$
is the map defined above.\end{lem}
\begin{proof}
First we observe that $\alpha$ and $\beta$ fix $e_{2}=[y,x]$. Thus,
by the second part of Proposition \ref{prop:BER-F2}, we have:
\[
Inn(\hat{F}_{2})\cap\textrm{Im}\rho\subseteq Z_{Inn\left(\hat{F}_{2}\right)}\left(Inn\left([y,x]\right)\right)=\overline{\left\langle Inn\left([y,x]\right)\right\rangle }=\overline{\left\langle Inn\left(e_{2}\right)\right\rangle }.
\]

Now, as $e_{2}\in\ker\pi$, where $\pi$ is as defined in the proof
of Lemma \ref{lem:main-lemma}, the image of $\overline{\left\langle Inn\left(e_{2}\right)\right\rangle }$
in $Inn\left(\widehat{\left\langle \alpha,\beta\right\rangle }\right)$
is trivial. Thus, the image of $Inn(\hat{F}_{2})\cap\textrm{Im}\rho$
in $Inn\left(\widehat{\left\langle \alpha,\beta\right\rangle }\right)$
is trivial , and isomorphic to $Inn(\hat{F}_{2})\cap\textrm{Im}\rho$
as we saw that $\textrm{Im}\rho$ is mapped isomorphically to $Inn\left(\widehat{\left\langle \alpha,\beta\right\rangle }\right)$.
So $Inn(\hat{F}_{2})\cap\textrm{Im}\rho$ is trivial. 
\end{proof}
This finishes the proof of Theorem \ref{thm:Free 2}. In \cite{key-5},
the authors give an explicit construction of a congruence subgroup
which is contained in a given finite index subgroup of $Aut(\hat{F}_{2})$.
They prove the following theorem:
\begin{thm}
\textup{\label{thm:explicit-BER}(cf. \cite{key-5}, Theorem 5.1)
}Let $H$ be a finite index normal subgroup of $G=Aut\left(F_{2}\right)$
such that $Inn\left(F_{2}\right)\leq H\leq Aut'\left(F_{2}\right)$
and let $n=\left[Aut'\left(F_{2}\right):H\right]$. Pick two distinct
odd primes $p,q\nmid n$, and set $m=n\cdot p^{n+1}$. Then, there
exists an explicitly constructed normal subgroup $M\leq F_{2}$ of
index dividing $144\cdot m^{4}\cdot q^{36\cdot m^{4}+1}$ such that
$G\left(M\right)\leq H$, when for a general normal subgroup $M\vartriangleleft F_{2}$
we define:
\[
G\left(M\right)=\left\{ \sigma\in G\,|\,\sigma\left(M\right)=M,\,\,\sigma\,\,acts\,\,trivially\,\,on\,\,F_{2}/M\right\} 
\]

\end{thm}
We end this section with a much simpler explicit construction of a
congruence subgroup and with a better bound for the index of $M$.
But before, let us recall the ``discrete version'' of Proposition
\ref{prop:BER-F2} from \cite{key-5}:
\begin{prop}
\textup{\label{prop:BER-dis-ver}(cf. \cite{key-5}, Propositions
2.2. and 2.6.) }Let $F$ be the free group on the set $X$, $\left|X\right|\geq2$,
and let $F/N$ be a finite quotient of $F$. Pick a prime $p$ not
dividing the order of $F/N$ and set $M=N^{p}N'.$ Then:

1. The image of every normal abelian subgroup of $F/M$ through the
natural projection $F/M\to F/N$, is trivial.

2. If $N\subseteq F'_{2}F_{2}^{6}$, $x,y\in X$, $x\neq y$, then
the image of the centralizer $Z_{F/M}\left([y,x]\cdot M\right)$ through
the natural projection $F/M\to F/N$, is $\left\langle [y,x]\cdot N\right\rangle $.\end{prop}
\begin{thm}
\label{thm:explicit}Let $H$ be a finite index normal subgroup of
$G=Aut\left(F_{2}\right)$ such that $Inn\left(F_{2}\right)\leq H\leq Aut'\left(F_{2}\right)=Inn\left(F_{2}\right)\rtimes\left\langle \alpha,\beta\right\rangle $
and let $n=\left[Aut'\left(F_{2}\right):H\right]$. Then for every
prime $p\nmid6n$, there exists an explicitly constructed normal subgroup
$M\leq F_{2}$ of index dividing $144\cdot n^{4}\cdot p^{36\cdot n^{4}+1}$
such that $G\left(M\right)\leq H$.\end{thm}
\begin{proof}
Recall the map $\pi:F_{2}\supseteq\Delta\to\left\langle \alpha,\beta\right\rangle $
from the proof of Lemma \ref{lem:main-lemma}, and let $t_{1}=1,\,t_{2}=x,\,t_{3}=y,\,t_{4}=xy$
be the system of representatives of right cosets of $\Delta$ in $F_{2}$.
Denote also $K=H\cap\left\langle \alpha,\beta\right\rangle $ and
define:
\begin{eqnarray*}
N & = & F_{2}'F_{2}^{6}\bigcap_{g\in F_{2}}g^{-1}\pi^{-1}\left(K\right)g=F_{2}'F_{2}^{6}\bigcap_{i=1}^{4}t_{i}^{-1}\pi^{-1}\left(K\right)t_{i}\\
M & = & F_{2}'F_{2}^{4}\cap N'N^{p}
\end{eqnarray*}
Then $\pi^{-1}\left(K\right)$ is a subgroup of index $n$ in $\Delta$
and $\bigcap_{i=1}^{4}t_{i}^{-1}\pi^{-1}\left(K\right)t_{i}$ is a
normal subgroup of $F_{2}$ of index dividing $n^{4}$ in $\Delta$,
and of index dividing $4n^{4}$ in $F_{2}$. So as $F'_{2}F_{2}^{6}$
is of index $9$ in $\Delta$, $N$ is a normal subgroup of index
dividing $36\cdot n^{4}$ in $F_{2}$. Thus, by the Schreier formula,
the index of $N'N^{p}$ in $F_{2}$ divides $36\cdot n^{4}\cdot p^{36\cdot n^{4}+1}$
and the index of $M$ in $F_{2}$ is dividing $4\cdot36\cdot n^{4}\cdot p^{36\cdot n^{4}+1}$.
So it remains to show that $G\left(M\right)\leq H$.

Let $\sigma\in G\left(M\right)$. As $M\leq F_{2}'F_{2}^{4}$ we have:
\[
G\left(M\right)\leq\ker\left(G\to Aut\left(F_{2}/\left(F_{2}'F_{2}^{4}\right)\right)\right)\leq Aut'\left(F_{2}\right)=Inn\left(F_{2}\right)\rtimes\left\langle \alpha,\beta\right\rangle 
\]
and therefore we can write $\sigma=Inn\left(f\right)\cdot\delta$
for some $f\in F_{2}$ and $\delta\in\left\langle \alpha,\beta\right\rangle $.
By assumption, $\sigma$ acts trivially on $F_{2}/M$ and thus $\delta$
acts on $F_{2}/M$ as $Inn\left(f^{-1}\right)$. Now, as $\alpha$
and $\beta$ fix $\left[y,x\right]$, we deduce that so does $\delta$.
Thus, $f\cdot M\in Z_{F_{2}/M}\left(\left[y,x\right]\cdot M\right)$
and by Proposition \ref{prop:BER-dis-ver}, $f\cdot N\in\left\langle [y,x]\cdot N\right\rangle $.
Hence, $\delta$ acts on the group $F_{2}/M$ as $Inn\left([y,x]^{r}\cdot n\right)$
for some $r\in\mathbb{Z}$ and $n\in N$. Therefore, $\delta$ acts
on $\Delta/M$ as $Inn\left(e_{2}^{r}\cdot n\right)$ for some $r\in\mathbb{Z}$
and $n\in N$. So, $\delta$ acts on $\pi\left(\Delta\right)/\pi\left(M\right)=\Delta/\left(M\cdot\ker\pi\right)$
as $Inn\left(\pi\left(e_{2}^{r}\cdot n\right)\right)$ for some $r\in\mathbb{Z}$
and $n\in N$. But $e_{2}\in\ker\pi$, so $\delta$ acts on $\pi\left(\Delta\right)/\pi\left(M\right)$
as $Inn\left(\pi\left(n\right)\right)$ for some $n\in N$. Now, by
the definition of $N$, $\pi\left(N\right)\subseteq K$ and also $\pi\left(M\right)\subseteq K'K^{p}$,
so $\delta$ acts on $\pi\left(\Delta\right)/K'K^{p}$ as $Inn\left(k\right)$
for some $k\in K$. Moreover, by the definition of $\pi$ we have
$\pi\left(\Delta\right)=\left\langle \alpha,\beta\right\rangle $
and by the computations we made in the proof of Lemma \ref{lem:main-lemma},
$\delta$ acts on $\left\langle \alpha,\beta\right\rangle $ as $Inn\left(\delta\right)$.
Thus, there exists some $k\in K$ such that $Inn\left(\delta\right)\cdot Inn\left(k\right)^{-1}$
acts trivially on $\left\langle \alpha,\beta\right\rangle /K'K^{p}$,
i.e. $\delta\cdot k^{-1}\in Z\left(\left\langle \alpha,\beta\right\rangle /K'K^{p}\right)$.
Now, by the first part of Proposition \ref{prop:BER-dis-ver}, as
$Z\left(\left\langle \alpha,\beta\right\rangle /K'K^{p}\right)$ is
an abelian normal subgroup of $\left\langle \alpha,\beta\right\rangle /K'K^{p}$
it is mapped trivially to $\left\langle \alpha,\beta\right\rangle /K$
. I.e. $\delta\cdot k^{-1}\in K$, so also $\delta\in K\subseteq H$.
Thus, $\sigma=Inn\left(f\right)\cdot\delta\in H$, as required.
\end{proof}

\section{The CSP for $\Phi_{2}$ \label{sec:CSP meta 2}}

In this section we will prove Theorem \ref{thm:.metabelian 2}, and
will show that the congruence kernel of the free metabelian group
on two generators is the free profinite group on a countable number
of generators. 

Before we start, let us observe that for a group $\Gamma$, one can
also ask a parallel congruence subgroup problem for $G=Out\left(\Gamma\right)$.
I.e. one can ask whether every finite index subgroup of $G$ contains
a principal congruence subgroup of the form: 
\[
G\left(M\right)=\ker\left(G\to Out\left(\Gamma/M\right)\right)
\]
for some finite index characteristic subgroup $M\leq\Gamma$. When
$\Gamma$ is finitely generated, this is equivalent to the question
whether the congruence map $\widehat{G}\to Out(\hat{\Gamma})$ is
injective. Moreover, it is easy to see that Lemma \ref{lem:cong-sub}
has a parallel version for $G$, namely, if $H\leq G$ is a congruence
subgroup of $G$, then:
\[
\ker(\widehat{G}\to Out(\hat{\Gamma}))=\ker(\widehat{H}\to Out(\hat{\Gamma}))
\]

We start now with the next proposition which is slightly more general
than Lemma 3.1. in \cite{key-5}. Nevertheless, it is proven by the
same arguments:
\begin{prop}
\textup{\label{prop:Out-Aut}(cf. \cite{key-5}, Lemma 3.1.) }Let
$\Gamma$ be a finitely generated residually finite group such that
$\hat{\Gamma}$ has a trivial center. Considering the congruence map
$\widehat{Out\left(\Gamma\right)}\to Out(\hat{\Gamma})$, we have:
\[
C\left(\Gamma\right)=\ker(\widehat{Aut\left(\Gamma\right)}\to Aut(\hat{\Gamma}))\cong\ker(\widehat{Out\left(\Gamma\right)}\to Out(\hat{\Gamma}))
\]

\end{prop}
It is well known that $\Phi_{2}$ is a residually finite group (cf.
\cite{key-6}, Theorem 2.11). It is also proven there that $Z(\hat{\Phi}_{2})$
is trivial (proposition 2.10). So by the above proposition:
\[
C\left(\Phi_{2}\right)=\ker(\widehat{Aut\left(\Phi_{2}\right)}\to Aut(\hat{\Phi}_{2}))\cong\ker(\widehat{Out\left(\Phi_{2}\right)}\to Out(\hat{\Phi}_{2}))
\]

In addition, it is an old result by Bachmuth \cite{key-13} that the
kernel of the surjective map: 
\[
\ker\left(Aut\left(\Phi_{2}\right)\to Aut\left(\Phi_{2}/\Phi'_{2}\right)=Aut\left(\mathbb{Z}^{2}\right)=GL_{2}\left(\mathbb{Z}\right)\right)=Inn\left(\Phi_{2}\right)
\]
i.e., $Out\left(\Phi_{2}\right)\cong GL_{2}\left(\mathbb{Z}\right)$.
Now, the free group $\left\langle \left(\begin{array}{cc}
1 & 2\\
0 & 1
\end{array}\right),\left(\begin{array}{cc}
1 & 0\\
2 & 1
\end{array}\right)\right\rangle $ is a congruence subgroup of $Out\left(\Phi_{2}\right)$ as it contains:
\[
\ker\left(Out\left(\Phi_{2}\right)\to Out\left(\Phi_{2}/\Phi_{2}'\Phi_{2}^{4}\right)\right)=\ker\left(GL_{2}\left(\mathbb{Z}\right)\to GL_{2}\left(\mathbb{Z}/4\mathbb{Z}\right)\right).
\]
So by the appropriate version of Lemma \ref{lem:cong-sub} and by
Proposition \ref{prop:Out-Aut}, we obtain that:
\begin{eqnarray*}
C\left(\Phi_{2}\right) & = & \ker(\widehat{Out\left(\Phi_{2}\right)}\to Out(\hat{\Phi}_{2}))\\
 & = & \ker(\widehat{GL_{2}\left(\mathbb{Z}\right)}\to Out(\hat{\Phi}_{2}))\\
 & = & \ker\left(\widehat{\left\langle \left(\begin{array}{cc}
1 & 2\\
0 & 1
\end{array}\right),\left(\begin{array}{cc}
1 & 0\\
2 & 1
\end{array}\right)\right\rangle }\to Out(\hat{\Phi}_{2})\right)
\end{eqnarray*}

Now, as $\left\langle \left(\begin{array}{cc}
1 & 2\\
0 & 1
\end{array}\right),\left(\begin{array}{cc}
1 & 0\\
2 & 1
\end{array}\right)\right\rangle $ is a free group, we can also state that:
\begin{eqnarray}
C\left(\Phi_{2}\right) & = & \ker\left(\widehat{\left\langle \left(\begin{array}{cc}
1 & 2\\
0 & 1
\end{array}\right),\left(\begin{array}{cc}
1 & 0\\
2 & 1
\end{array}\right)\right\rangle }\to Out(\hat{\Phi}_{2})\right)\label{eq:(*)}\\
 & \cong & \ker(\widehat{\left\langle \alpha,\beta\right\rangle }\to Aut(\hat{F}_{2})\to Aut(\hat{\Phi}_{2})\to Out(\hat{\Phi}_{2}))\nonumber 
\end{eqnarray}
where $\alpha$ and $\beta$ are the automorphisms of $F_{2}$ that
we defined in the previous section, which are the preimages of $\left(\begin{array}{cc}
1 & 2\\
0 & 1
\end{array}\right)$ and $\left(\begin{array}{cc}
1 & 0\\
2 & 1
\end{array}\right)$ under the map $Aut\left(F_{2}\right)\to GL_{2}\left(\mathbb{Z}\right)$,
respectively. So all we need to show is that:
\begin{lem}
\label{lem:.main-meta-2}$C\left(\Phi_{2}\right)=\ker(\widehat{\left\langle \alpha,\beta\right\rangle }\to Out(\hat{\Phi}_{2}))=\hat{F}_{\omega}$.\end{lem}
\begin{proof}
As the free group $\left\langle \left(\begin{array}{cc}
1 & 2\\
0 & 1
\end{array}\right),\left(\begin{array}{cc}
1 & 0\\
2 & 1
\end{array}\right)\right\rangle $ is a congruence subgroup of the group $Aut\left(\mathbb{Z}^{2}\right)=Out\left(\mathbb{Z}^{2}\right)=GL_{2}\left(\mathbb{Z}\right)$,
we have:
\begin{eqnarray*}
C\left(\mathbb{Z}^{2}\right) & = & \ker\left(\widehat{\left\langle \left(\begin{array}{cc}
1 & 2\\
0 & 1
\end{array}\right),\left(\begin{array}{cc}
1 & 0\\
2 & 1
\end{array}\right)\right\rangle }\to Out(\mathbb{\hat{Z}}^{2})\right)\\
 & = & \ker(\widehat{\left\langle \alpha,\beta\right\rangle }\to Out(\mathbb{\hat{Z}}^{2}))\\
 & = & \ker(\widehat{\left\langle \alpha,\beta\right\rangle }\to Aut(\hat{\Phi}_{2})\to Out(\hat{\Phi}_{2})\to Out(\mathbb{\hat{Z}}^{2})=Aut(\mathbb{\hat{Z}}^{2}))
\end{eqnarray*}
Thus, if we denote: $C=\ker(\widehat{\left\langle \alpha,\beta\right\rangle }\to Aut(\hat{\Phi}_{2}))$,
then using equation (\ref{eq:(*)}), we have: $C\leq C\left(\Phi_{2}\right)\leq C\left(\mathbb{Z}^{2}\right)$.
Now, if we consider the action of $\hat{\Phi}_{2}$ on $\overline{\Phi_{2}'}=\ker(\hat{\Phi}_{2}\to\hat{\mathbb{Z}}^{2})$
by conjugation, then as $\overline{\Phi_{2}'}$ is abelian, we actually
obtain an action on $\overline{\Phi_{2}'}$ as a $\mathbb{Z}[\hat{\Phi}_{2}/\overline{\Phi_{2}'}]=\mathbb{Z}[\hat{\mathbb{Z}}^{2}]$-module,
which is generated by the element $\left[y,x\right]$ as a $\mathbb{Z}[\hat{\mathbb{Z}}^{2}]$-module,
since $\left\langle x,y\,|\,\left[y,x\right]=1\right\rangle $ is
a presentation of $\mathbb{Z}^{2}$. Moreover, as we observed previously,
$\alpha$ and $\beta$ fix $\left[y,x\right]$. Therefore, $C\left(\mathbb{Z}^{2}\right)$
acts trivially not only on $\hat{\Phi}_{2}/\overline{\Phi_{2}'}=\hat{\mathbb{Z}}^{2}$
but also on $\overline{\Phi_{2}'}$. 

Let us now make the following observation: if $\sigma,\tau$ are two
automorphisms of a group $\Gamma$ which act trivially on $\Gamma/M$
and on $M$, where $M\vartriangleleft\Gamma$ is abelian, then $\sigma$
and $\tau$ commute. Indeed, if $g\in\Gamma$, then $\sigma\left(g\right)=g\cdot m$
and $\tau\left(g\right)=g\cdot n$ for some $m,n\in M$, and thus:
\[
\tau\left(\sigma\left(g\right)\right)=\tau\left(g\cdot m\right)=g\cdot n\cdot m=g\cdot m\cdot n=\sigma\left(g\cdot n\right)=\sigma\left(\tau\left(g\right)\right).
\]

The conclusion from the above observation and from the previous discussion
is that $C\left(\mathbb{Z}^{2}\right)/C$ is abelian, and thus, $C\left(\mathbb{Z}^{2}\right)/C\left(\Phi_{2}\right)$
is also abelian. Finally, $C\left(\mathbb{Z}^{2}\right)$ is known
to be isomorphic to $\hat{F}_{\omega}$ (\cite{key-17}, \cite{key-4}).
Moreover, by Proposition 1.10 and Corollary 3.9 of \cite{key-8} every
normal closed subgroup $N$ of $\hat{F}_{\omega}$ such that $\hat{F}_{\omega}/N$
is abelian, is also isomorphic to $\hat{F}_{\omega}$. Thus, $C\left(\Phi_{2}\right)\cong\hat{F}_{\omega}$
as well, as required.\end{proof}
\begin{rem}
Our proof of Theorem \ref{thm:.metabelian 2} is much shorter than
the one given in \cite{key-6}, but the latter gives more information.
We show here that $C\left(\mathbb{Z}^{2}\right)/C\left(\Phi_{2}\right)$
is abelian, while from \cite{key-6} one can deduce that, infact,
$C\left(\Phi_{2}\right)=C\left(\mathbb{Z}^{2}\right)$. See $\varoint$\ref{sec:Remarks}
fore more.
\end{rem}

\section{\label{sec:CSP for 3}The CSP for $\Phi_{3}$}

In this section we will prove Theorem \ref{thm: metabelian 3} which
claims that $C\left(\Phi_{3}\right)$ contains a copy of $\hat{F}_{\omega}$.
Let us start by showing that $Aut\left(\Phi_{3}\right)$ is large:
\begin{prop}
The group $Aut\left(\Phi_{3}\right)$ is large, i.e. it has a finite
index subgroup that can be mapped onto a non-abelian free group.\end{prop}
\begin{proof}
The proof will follow the method developed in \cite{key-13-1} to
produce arithmetic quotients of $Aut\left(F_{n}\right)$. Denote the
free group on $3$ generators by $F_{3}=\left\langle x,y,z\right\rangle $,
and the cyclic group of order $2$ by $C_{2}=\left\{ 1,g\right\} $.
Define the map $\pi:F_{3}\to C_{2}$ by: $\pi=\begin{cases}
x & \mapsto g\\
y,z & \mapsto1
\end{cases}$, and its kernel by $R=\ker\pi$. Then, using the right transversal
$T=\left\{ 1,x\right\} $, we deduce by Theorem \ref{thm:Schreier}
that $R$ is freely generated by: $x^{2},\,y,\,xyx^{-1},\,z,\,xzx^{-1}$.
Thus, $\bar{R}=R/R'=\mathbb{Z}^{5}$ is generated as a free abelian
group by the images: 
\[
v_{1}=\overline{x^{2}},\,v_{2}=\overline{y},\,v_{3}=\overline{xyx^{-1}},\,v_{4}=\overline{z},\,v_{5}=\overline{xzx^{-1}}
\]
Now, the action of $F_{3}$ on $R$ by conjugation induces an action
of $F_{3}/R=C_{2}=\left\{ 1,g\right\} $ on $\bar{R}=R/R'$, sending:
\[
g\mapsto\begin{cases}
v_{1}=\overline{x^{2}}\mapsto\overline{x^{2}} & =v_{1}\\
v_{2}=\overline{y}\mapsto\overline{x^{-2}\left(xyx^{-1}\right)x^{2}}=\overline{xyx^{-1}} & =v_{3}\\
v_{3}=\overline{xyx^{-1}}\mapsto\overline{y} & =v_{2}\\
v_{4}=\overline{z}\mapsto\overline{x^{-2}\left(xzx^{-1}\right)x^{2}}=\overline{xzx^{-1}} & =v_{5}\\
v_{5}=\overline{xzx^{-1}}\mapsto\bar{z} & =v_{4}
\end{cases}\,\,\,=\left(\begin{array}{ccccc}
1 & 0 & 0 & 0 & 0\\
0 & 0 & 1 & 0 & 0\\
0 & 1 & 0 & 0 & 0\\
0 & 0 & 0 & 0 & 1\\
0 & 0 & 0 & 1 & 0
\end{array}\right)=B
\]
The above matrix has two eigenvalues $\lambda=\pm1$ and the eigenspaces
are:
\begin{eqnarray*}
V_{1} & = & Sp\left\{ v_{1},v_{2}+v_{3},v_{4}+v_{5}\right\} \\
V_{-1} & = & Sp\left\{ v_{2}-v_{3},v_{4}-v_{5}\right\} 
\end{eqnarray*}

Recall, $\Phi_{3}=F_{3}/F''_{3}$, and as $F_{3}/R$ is abelian, $F_{3}/R'$
is metabelian. Thus, we have a surjective homomorphism: $\Phi_{3}\twoheadrightarrow F_{3}/R'$.
Denote now: $S=R/F''_{3}$, so we can identify: $F_{3}/R\cong\Phi_{3}/S$,
$F_{3}/R'\cong\Phi_{3}/S'$ and $\bar{R}=R/R'\cong S/S'=\bar{S}$.
So as before, $\Phi_{3}/S=C_{2}$ acts on $\bar{S}$ by the matrix
$B$.

Denote also $G\left(S\right)=\left\{ \sigma\in Aut\left(\Phi_{3}\right)\,|\,\sigma\left(S\right)=S\right\} $.
It is clear that $G\left(S\right)$ is of finite index in $Aut\left(\Phi_{3}\right)$
with a natural map: $G\left(S\right)\to Aut\left(S\right)$ which
induces a map: $\rho:G\left(S\right)\to Aut\left(\bar{S}\right)=GL_{5}\left(\mathbb{Z}\right)$.
We claim now that if $\sigma\in G\left(S\right)$ than $\rho\left(\sigma\right)$
commutes with $B$. First observe that there exists some $s\in S$
such that $\sigma\left(x\right)=sx$ ($x$ now plays the role of the
image of $x$ under the map $F_{3}\to\Phi_{3}$). Now, let $t\in S$,
and remember that the action of $B$ on $\bar{S}$ is induced by the
action of $x$ on $S$ by conjugation. So:
\begin{eqnarray*}
\sigma\left(x^{-1}tx\right) & = & \sigma\left(x\right)^{-1}\sigma\left(t\right)\sigma\left(x\right)=\\
 & = & x^{-1}s^{-1}\sigma\left(t\right)sx=\\
 & = & \left(x^{-1}sx\right)^{-1}\left(x^{-1}\sigma\left(t\right)x\right)\left(x^{-1}sx\right)
\end{eqnarray*}
and hence:
\begin{eqnarray*}
\left(\rho\left(\sigma\right)\cdot B\right)\left(\bar{t}\right) & = & \overline{\sigma\left(x^{-1}tx\right)}=\\
 & = & \overline{\left(x^{-1}sx\right)^{-1}\left(x^{-1}\sigma\left(t\right)x\right)\left(x^{-1}sx\right)}=\\
 & = & \overline{x^{-1}\sigma\left(t\right)x}=\left(B\cdot\rho\left(\sigma\right)\right)\left(\bar{t}\right)
\end{eqnarray*}
Therefore, $\rho\left(G\left(S\right)\right)$ commutes with $B$.
It follows that the eigenspaces of $B$ are invriant under the action
of $G\left(S\right)$. Inparticular, we deduce that $V_{-1}$ is invariant
under the action of $\rho\left(G\left(S\right)\right)$. Thus, we
obtain a homomorphism $\nu:G\left(S\right)\to Aut\left(V_{-1}\cap\bar{S}\right)=GL_{2}\left(\mathbb{Z}\right)$.

Consider now the following automorphisms of $Aut\left(\Phi_{3}\right)$
($x,\,y,\,z$ play the role of the images of $x,\,y,\,z$ under $F_{3}\to\Phi_{3}$):
\[
\alpha=\begin{cases}
x\mapsto & x\\
y\mapsto & y\\
z\mapsto & zy
\end{cases}\,,\,\,\,\beta=\begin{cases}
x\mapsto & x\\
y\mapsto & yz\\
z\mapsto & z
\end{cases}
\]

So $\alpha,\beta\in G\left(S\right)$ act on $V_{-1}=Sp\left\{ u_{1}=v_{2}-v_{3},\,u_{2}=v_{4}-v_{5}\right\} $
in the following way:
\begin{eqnarray*}
\alpha\left(u_{1}\right) & = & \alpha\left(\overline{y}-\overline{xyx^{-1}}\right)=\overline{y}-\overline{xyx^{-1}}=u_{1}\\
\alpha\left(u_{2}\right) & = & \alpha\left(\overline{z}-\overline{xzx^{-1}}\right)=\overline{z}+\bar{y}-\overline{xzx^{-1}}-\overline{xyx^{-1}}=u_{2}+u_{1}
\end{eqnarray*}
\begin{eqnarray*}
\beta\left(u_{1}\right) & = & \beta\left(\overline{y}-\overline{xyx^{-1}}\right)=\bar{y}+\overline{z}-\overline{xyx^{-1}}-\overline{xzx^{-1}}=u_{1}+u_{2}\\
\beta\left(u_{2}\right) & = & \beta\left(\overline{z}-\overline{xzx^{-1}}\right)=\overline{z}-\overline{xzx^{-1}}=u_{2}
\end{eqnarray*}

Therefore, under the map $\nu:G\left(S\right)\to GL_{2}\left(\mathbb{Z}\right)$
we have: $\alpha\mapsto\left(\begin{array}{cc}
1 & 1\\
0 & 1
\end{array}\right)$ and $\beta\mapsto\left(\begin{array}{cc}
1 & 0\\
1 & 1
\end{array}\right)$. Thus, the image of $G\left(S\right)$ contains $\left\langle \left(\begin{array}{cc}
1 & 2\\
0 & 1
\end{array}\right),\left(\begin{array}{cc}
1 & 0\\
2 & 1
\end{array}\right)\right\rangle $ which is free and of finite index in $GL_{2}\left(\mathbb{Z}\right)$.
Finally, if we denote the preimage $H=\nu^{-1}\left(\left\langle \left(\begin{array}{cc}
1 & 2\\
0 & 1
\end{array}\right),\left(\begin{array}{cc}
1 & 0\\
2 & 1
\end{array}\right)\right\rangle \right)$, then $H$ is a finite index subgroup of $Aut\left(\Phi_{3}\right)$
that can be mapped onto a free group, as required.
\end{proof}
Let us now continue with the following definition:
\begin{defn}
We say that a group $P$ is involved in a group $Q$, if it isomorphic
to a quotient group of some subgroup of $Q$.
\end{defn}
It is not difficult to see that if a finite group $P$ is involved
in a profinite group $Q$, than it is involved in a finite quotient
of $Q$. Now, we showed that $Aut\left(\Phi_{3}\right)$ has a finite
index subgroup $H$ which can be mapped onto $F_{2}$. Thus we have
a map: $\widehat{H}\twoheadrightarrow\hat{F}_{2}$, but as $\hat{F}_{2}$
is free, the map splits, and thus $\widehat{H}$ and hence $\widehat{Aut\left(\Phi_{3}\right)}$,
contains a copy of $\hat{F}_{2}$. Thus, any finite group is involved
in $\widehat{Aut\left(\Phi_{3}\right)}$. On the other hand, we claim:
\begin{prop}
Let $P$ be a non-abelian finite simple group which is involved in
$Aut(\hat{\Phi}_{3})$. Then, for some prime $p$ and some $d\in\mathbb{N}$,
$P$ is involved in $SL_{3}\left(p^{d}\right)$, the special linear
group over the field of order $p^{d}$.\end{prop}
\begin{proof}
Let $F_{n}$ be the free group on $x_{1},\ldots,x_{n}$. Then there
is a natural injective homomorphism from $F_{n}$ into the matrix
group:
\[
\left\{ \left(\begin{array}{cc}
g & 0\\
t & 1
\end{array}\right)\,|\,g\in F_{n},\,t\in\sum_{i=1}^{n}\mathbb{Z}\left[F_{n}\right]t_{i}\right\} 
\]
defined by the map:
\[
x_{i}\mapsto\left(\begin{array}{cc}
x_{i} & 0\\
t_{i} & 1
\end{array}\right)\,,\,\,\,1\leq i\leq n
\]
where $t_{i}$ is a free basis for a right $\mathbb{Z}\left[F_{n}\right]$-module.
This is called the Magnus embedding. Usually, its properties are studied
by Fox's free differential calculus, but we will not need it here
explicitly (cf. \cite{key-26}, \cite{key-27}, \cite{key-25}). 

One can prove, by induction on its length, that for a word $w\in F_{n}$,
under the Magnus embedding, $w\mapsto\left(\begin{array}{cc}
w & 0\\
\sum_{i=1}^{n}w_{i}t_{i} & 1
\end{array}\right)$ where: 
\begin{equation}
w-1=\sum_{i=1}^{n}\left(x_{i}-1\right)w_{i}.\label{eq: magnus}
\end{equation}
The identity (\ref{eq: magnus}) shows that the polynomials $w_{i}$
determine the word $w$ uniquely. Thus, we have an injective map (which
is not homomorphism) $J:End\left(F_{n}\right)\to M_{n}\left(\mathbb{Z}\left[F_{n}\right]\right)$
defined by:
\[
\alpha\overset{J}{\mapsto}\left(\begin{array}{ccc}
\alpha\left(x_{1}\right)_{1} & \cdots & \alpha\left(x_{n}\right)_{1}\\
\vdots &  & \vdots\\
\alpha\left(x_{1}\right)_{n} & \cdots & \alpha\left(x_{n}\right)_{n}
\end{array}\right)
\]
It is not difficult to check, using the identity (\ref{eq: magnus}),
that the above map satisfies:
\[
J\left(\alpha\circ\beta\right)=J\left(\alpha\right)\cdot\alpha\left(J\left(\beta\right)\right)
\]
where by $\alpha\left(J\left(\beta\right)\right)$ we mean that $\alpha$
acts on every enrty of $J\left(\beta\right)$ separately. 

Now, for $m\in\mathbb{N}$, denote: $K_{n,m}=F_{n}^{m}F'_{n}$ and
$\mathbb{Z}_{m}=\mathbb{Z}/m\mathbb{Z}$. Then, the natural maps $F_{n}\to F_{n}/K_{n,m}=\mathbb{Z}_{m}^{n}$
and $\mathbb{Z}\to\mathbb{Z}_{m}$ induce a map:
\begin{eqnarray*}
\pi_{n,m}:F_{n} & \to & \left\{ \left(\begin{array}{cc}
g & 0\\
t & 1
\end{array}\right)\,|\,g\in F_{n},\,t\in\sum_{i=1}^{n}\mathbb{Z}\left[F_{n}\right]t_{i}\right\} \\
 & \to & \left\{ \left(\begin{array}{cc}
g & 0\\
t & 1
\end{array}\right)\,|\,g\in\mathbb{Z}_{m}^{n},\,t\in\sum_{i=1}^{n}\mathbb{Z}_{m}\left[\mathbb{Z}_{m}^{n}\right]t_{i}\right\} .
\end{eqnarray*}

It is shown in (\cite{key-6}, Proposition 2.6) that $\ker\left(\pi_{n,m}\right)=K_{n,m}^{m}K_{n,m}'$
and hence $\Phi_{n,m}:=\textrm{Im}\left(\pi_{n,m}\right)\cong F_{n}/K_{n,m}^{m}K'_{n,m}$.
Moreover, it is proven there (Proposition 2.7) that we have the following
equality:
\[
\hat{\Phi}_{n}=\underleftarrow{\lim}_{m}\Phi_{n,m}.
\]

Observe now that for every $m_{2}|m_{1}$, $\ker\left(\Phi_{n,m_{1}}\to\Phi_{n,m_{2}}\right)$
is characteristic in $\Phi_{n,m_{1}}$, and for every $m$, $\ker(\hat{\Phi}_{n}\to\Phi_{n,m})$
is characteristic in $\hat{\Phi}_{n}$. Thus:

\[
Aut(\hat{\Phi}_{n})=Aut(\underleftarrow{\lim}_{m}\Phi_{n,m})=\underleftarrow{\lim}_{m}Aut\left(\Phi_{n,m}\right).
\]

Now, observe that the identity (\ref{eq: magnus}) is also valid for
the entries of the elements of $\Phi_{n,m}$, and thus, every element
of $\Phi_{n,m}$ is determined by its left lower coordinate. Therefore,
as every automorphism of $\Phi_{n,m}$ can be lifted to an endomorphism
of $F_{n}$, we have an injective map (which is not a homomorphism)
$J_{m}:Aut\left(\Phi_{n,m}\right)\to M_{n}\left(\mathbb{Z}_{m}\left[\mathbb{Z}_{m}^{n}\right]\right)$
which satisfies the identity:
\[
J_{m}\left(\alpha\circ\beta\right)=J_{m}\left(\alpha\right)\cdot\alpha\left(J_{m}\left(\beta\right)\right)
\]
where the action of $\alpha$ on $\mathbb{Z}_{m}\left[\mathbb{Z}_{m}^{n}\right]=\mathbb{Z}_{m}\left[F_{n}/K_{n,m}\right]$
is through the natural projection $\Phi_{n,m}\cong F_{n}/K_{n,m}^{m}K'_{n,m}\to F_{n}/K_{n,m}\cong\mathbb{Z}_{m}^{n}$.

We denote now $KA\left(\Phi_{n,m}\right)=\ker\left(Aut\left(\Phi_{n,m}\right)\to Aut\left(\Phi_{n,m}/K_{n,m}\right)\right)$.
Observe, that as $KA\left(\Phi_{n,m}\right)$ acts trivially on $\Phi_{n,m}/K_{n,m}=\mathbb{Z}_{m}^{n}$,
the map $J_{m}$ gives us a $homomorphism$, which is also injective,
as mentioned above:
\[
J_{m}:KA\left(\Phi_{n,m}\right)\to GL_{n}\left(\mathbb{Z}_{m}\left[\mathbb{Z}_{m}^{n}\right]\right)
\]

Now, if $P$ is a non-abelian simple group which is involved in $Aut(\hat{\Phi}_{3})$,
then it must be involved in $Aut\left(\Phi_{3,m}\right)$ for some
$m$. Thus, it must be involved either in $Aut\left(\Phi_{3,m}/K_{3,m}\right)=GL_{3}\left(\mathbb{Z}_{m}\right)$
or in $KA\left(\Phi_{3,m}\right)\leq GL_{3}\left(\mathbb{Z}_{m}\left[\mathbb{Z}_{m}^{3}\right]\right)$.
So it must be involved in $GL_{3}\left(R\right)$ for some finite
commutative ring $R$. As every finite commutative ring is artinian,
it can be decomposed as:
\[
R=R_{1}\times\ldots\times R_{l}
\]
for some local finite rings $R_{1},\ldots,R_{l}$, so:
\[
GL_{3}\left(R\right)=GL_{3}\left(R_{1}\right)\times\ldots\times GL_{3}\left(R_{l}\right)
\]
and thus $P$ must be involved in $GL_{3}\left(R\right)$ for some
local finite commutative ring $R$. Denote the unique maximal ideal
of $R$ by $M\vartriangleleft R$. As $R$ is a finite local Noetherian
ring, it is well known that $M^{r}=0$ for some $r\in\mathbb{N}$. 

Note now that if $S,T\vartriangleleft R$ for some commutative ring
$R$, and 
\begin{eqnarray*}
I+A & \in & \ker\left(GL_{3}\left(R\right)\to GL_{3}\left(R/S\right)\right)\\
I+B & \in & \ker\left(GL_{3}\left(R\right)\to GL_{3}\left(R/T\right)\right)
\end{eqnarray*}
when $I$ denotes the identity element in $GL_{3}\left(R\right)$,
then 
\[
\left[I+A,I+B\right]\in\ker\left(GL_{3}\left(R\right)\to GL_{3}\left(R/ST\right)\right).
\]
Indeed, if $I+C=\left(I+A\right)^{-1}$ and $I+D=\left(I+B\right)^{-1}$
then, as $AB=CD=AD=CB=0\,\,\,(\textrm{mod}\,ST)$ we have:
\begin{eqnarray*}
\left[I+A,I+B\right] & = & \left(I+A\right)\left(I+B\right)\left(I+C\right)\left(I+D\right)=\\
 & = & I+AC+A+BD+B+C+D\,\,\,(\textrm{mod}\,ST)\\
 & = & I+\left(I+A\right)\left(I+C\right)-I+\left(I+B\right)\left(I+D\right)-I=I\,\,\,(\textrm{mod}\,ST)
\end{eqnarray*}

With the above observation we deduce that for every $k\geq1$, the
kernel of the map $GL_{3}\left(R/M^{k+1}\right)\to GL_{3}\left(R/M^{k}\right)$
is abelian. So, $P$ must be involved in $GL_{3}\left(R/M\right)=GL_{3}\left(p^{d}\right)$
for some prime $p$ and $d\in\mathbb{N}.$ Finally, using the fact
that $GL_{3}\left(p^{d}\right)/SL_{3}\left(p^{d}\right)$ is abelian,
we obtain that $P$ is involved in $SL_{3}\left(p^{d}\right)$, as
required.\end{proof}
\begin{cor}
\label{prop:not involved} There exists a finite simple group which
is not involved in $Aut(\hat{\Phi}_{3})$.\end{cor}
\begin{proof}
By the proposition above, it is enough to show that there is a finite
simple non-abelian group which is not involved in $SL_{3}\left(p^{d}\right)$
for any prime $p$ and $d\in\mathbb{N}$. Now, by a theorem of Jordan,
there exists an integer-valued function $J\left(n\right)$ such that
for every field $\mathbb{F}$ , $char\left(\mathbb{F}\right)=0$,
any finite subgroup of $GL_{n}\left(\mathbb{F}\right)$ contains a
normal abelian subgroup of index at most $J\left(n\right)$. As a
corollary of this theorem, Schur proved that the same holds (with
the same function) for any finite subgroup $Q\leq GL_{n}\left(\mathbb{F}\right)$
with $char\left(\mathbb{F}\right)=p>0$, provided $p\nmid\left|Q\right|$
(cf. \cite{key-2} chapter 9). Clearly, the same holds for any group
which is involved in such a finite group $Q$.

We claim that for $n$ large enough, $Alt\left(n\right)$ is not involved
in $SL_{3}\left(p^{d}\right)$ for any $p$ and $d$. Indeed, fix
two different primes $q_{1}$ and $q_{2}$ larger than $J\left(3\right)$.
Then, for $n$ sufficiently large (e.g. $n>q_{i}^{3}$) the $q_{i}$-sylow
subgroup $S_{i}$ of $Alt\left(n\right)$ is non-abelian (since $Alt\left(n\right)$
contains the non-abelian $q_{i}$-group of order $q_{i}^{3}$) and
every subgroup of $S_{i}$ of index $\leq J\left(3\right)$ is equal
to $S_{i}$, so also non-abelian. If $Alt\left(n\right)$ were involved
in $SL_{3}\left(p^{d}\right)$ then for at least one of the $q_{i}$,
$q_{i}\neq p$, a contradiction.\end{proof}
\begin{cor}
The congruence kernel $C\left(\Phi_{3}\right)$ contains a copy of
$\hat{F}_{\omega}$.\end{cor}
\begin{proof}
The immediate conclusion of Corollary \ref{prop:not involved} is
that $Aut(\hat{\Phi}_{3})$ does not contain a copy of $\hat{F}_{2}$.
Thus, the intersection of $C\left(\Phi_{3}\right)$ and the copy of
$\hat{F}_{2}$ in $\widehat{Aut\left(\Phi_{3}\right)}$ is not trivial.
Thus, $C\left(\Phi_{3}\right)$ contains a non-trivial normal closed
subgroup $N$ of $\hat{F}_{2}$. By Theorem 3.10 in \cite{key-8}
it contains a copy of $\hat{F}_{\omega}$, as required.
\end{proof}

\section{Remarks and open problems \label{sec:Remarks}}

We end this paper with several remarks and open problems. Denote the
free solvable group of derived length $r$ on $2$ generators by $\Phi_{2,r}$.
By combining the results of (\cite{key-9}, Theorem 1) and (\cite{key-10},
Theorem 1.4) we have:
\[
\ker\left(Aut\left(\Phi_{2,r}\right)\to Aut\left(\mathbb{Z}^{2}\right)=GL_{2}\left(\mathbb{Z}\right)\right)=Inn\left(\Phi_{2,r}\right)
\]
for every $r$, i.e. $Out\left(\Phi_{2,r}\right)=GL_{2}\left(\mathbb{Z}\right)$.
So by the same arguments as in $\varoint$\ref{sec:CSP meta 2} we
have:
\[
C\left(\Phi_{2,r}\right)=\ker(\widehat{\left\langle \alpha,\beta\right\rangle }\to Out(\hat{\Phi}_{2,r}))
\]
As $Out(\hat{\Phi}_{2,r+1})$ is mapped onto $Out(\hat{\Phi}_{2,r})$,
we obtain the sequence:
\begin{eqnarray*}
C\left(\mathbb{Z}^{2}\right) & = & C\left(\Phi_{2,1}\right)\geq C\left(\Phi_{2}\right)=C\left(\Phi_{2,2}\right)\geq C\left(\Phi_{2,3}\right)\geq\\
 & \geq & C\left(\Phi_{2,4}\right)\geq\ldots\geq C\left(\Phi_{2,r}\right)\geq\ldots\geq C\left(F_{2}\right)=\left\{ e\right\} 
\end{eqnarray*}
and a natural question is whether the inequalities are strict or not.
An equivalent reformulation of this question is the following: the
cosets of the kernels
\[
\ker(GL_{2}\left(\mathbb{Z}\right)=Out\left(\Phi_{2,r}\right)\to Out\left(\Phi_{2,r}/K\right))
\]
for characteristic finite index subgroups $K\leq\Phi_{2,r}$ provide
a basis for a topology $\mathfrak{\mathcal{\mathscr{C}}}\left(r\right)$
on $GL_{2}\left(\mathbb{Z}\right)$, called the congruence topology
with respect to $\Phi_{2,r}$, which is weaker (equal) than the profinite
topology $\mathscr{F}$ of $GL_{2}\left(\mathbb{Z}\right)$, and stronger
(equal) than the classical congruence topology of $GL_{2}\left(\mathbb{Z}\right)$.
The latter is equal to $\mathfrak{\mathcal{\mathscr{C}}}\left(1\right)$.
So, the question above is equivalent to the question whether these
topologies are strictly weaker than $\mathscr{F}$, and whether the
topology $\mathfrak{\mathcal{\mathscr{C}}}\left(r\right)$, for a
given $r$, is strictly weaker than $\mathfrak{\mathcal{\mathscr{C}}}\left(r+1\right)$.

For example, Theorem \ref{thm:Free 2} which states that $C\left(F_{2}\right)=\left\{ e\right\} $
is equivalent to the statement that the congruence topology which
$Out(\hat{F}_{2})$ induces on $Out\left(F_{2}\right)=GL_{2}\left(\mathbb{Z}\right)$
is equal to the profinite topology of $GL_{2}\left(\mathbb{Z}\right)$.

Considering Theorem \ref{thm:.metabelian 2} we deduce that $\mathfrak{\mathcal{\mathscr{C}}}\left(2\right)\subsetneqq\mathscr{F}$,
but with the proof we gave here one can not decide whether $\mathfrak{\mathcal{\mathscr{C}}}\left(1\right)=\mathfrak{\mathcal{\mathscr{C}}}\left(2\right)$
or $\mathfrak{\mathcal{\mathscr{C}}}\left(1\right)\subsetneqq\mathfrak{\mathcal{\mathscr{C}}}\left(2\right)$.
Equivalently, we can not decide whether $C\left(\mathbb{Z}^{2}\right)=C\left(\Phi_{2}\right)$
or $C\left(\mathbb{Z}^{2}\right)\gneqq C\left(\Phi_{2}\right)$. But,
in \cite{key-6} it was shown quite surprisingly, that:
\begin{thm}
$\mathfrak{\mathcal{\mathscr{C}}}\left(1\right)=\mathfrak{\mathcal{\mathscr{C}}}\left(2\right)$,
or equivalently $C\left(\mathbb{Z}^{2}\right)=C\left(\Phi_{2}\right)$.
\end{thm}
The proof in \cite{key-6} suggested to conjecture that $\mathfrak{\mathcal{\mathscr{C}}}\left(1\right)=\mathfrak{\mathcal{\mathscr{C}}}\left(2\right)=\mathfrak{\mathcal{\mathscr{C}}}\left(r\right)$
for every $r$. But, the explicit construction of a congruence subgroup
we gave in $\varoint$\ref{sec:SCP F2} gives a counter example:
\begin{prop}
$\mathfrak{\mathcal{\mathscr{C}}}\left(1\right)\subsetneqq\mathfrak{\mathcal{\mathscr{C}}}\left(r\right)$
for every $r\geq3$. Equivalently $C\left(\mathbb{Z}^{2}\right)\gneqq C\left(\Phi_{2,r}\right)$
for every $r\geq3$.\end{prop}
\begin{proof}
Denote $G=\left\langle \left(\begin{array}{cc}
1 & 2\\
0 & 1
\end{array}\right),\left(\begin{array}{cc}
1 & 0\\
2 & 1
\end{array}\right)\right\rangle \leq GL_{2}\left(\mathbb{Z}\right)$. Then by a theorem of Reiner \cite{key-14}, for every $p\neq2$,
$G'G^{p}$ is not a congruence subgroup of $GL_{2}\left(\mathbb{Z}\right)$
in the classical manner, i.e. $G'G^{p}\notin\mathfrak{\mathcal{\mathscr{C}}}\left(1\right)$.
On the other hand, applying the explicit construction given in Theorem
\ref{thm:explicit}, we obtain a finite index normal subgroup $M\vartriangleleft F_{2}$
such that $F_{2}/M$ is of solvability length $3$ such that\footnote{We remark that if one wants $M$ to be characteristic, all we need
to do, is to replace $M$ by $\bigcap_{\sigma\in Aut\left(F_{2}\right)}\sigma\left(M\right)$,
and this procedure does not change the solvability length of $F_{2}/M$.}: 
\[
\ker\left(Out\left(F_{2}\right)=GL_{2}\left(\mathbb{Z}\right)\to Out\left(F_{2}/M\right)\right)\leq G'G^{p}.
\]
This shows that $G'G^{p}$ is a congruence subgroup of $GL_{2}\left(\mathbb{Z}\right)$
with respect to the congruence topology induced by $Out(\hat{\Phi}_{2,3})$.
Equivalently, $\mathfrak{\mathcal{\mathscr{C}}}\left(1\right)\subsetneqq\mathfrak{\mathcal{\mathscr{C}}}\left(3\right)$
or $C\left(\mathbb{Z}^{2}\right)\gneqq C\left(\Phi_{2,3}\right)$,
as required.
\end{proof}
The proposition suggests the following conjecture:
\begin{conjecture}
$C\left(\Phi_{2,r}\right)\gneqq C\left(\Phi_{2,r+1}\right)$ for every
$r\geq2$, or equivalently $\mathfrak{\mathcal{\mathscr{C}}}\left(r\right)\subsetneqq\mathfrak{\mathcal{\mathscr{C}}}\left(r+1\right)$.
In particular, $C\left(\Phi_{2,r}\right)\neq\left\{ e\right\} =C\left(F_{2}\right)$
and $\mathfrak{\mathcal{\mathscr{C}}}\left(r\right)\neq\mathscr{F}$
for every $r$.
\end{conjecture}
We should remark that we do not even know to decide whether $C\left(\Phi_{2,r}\right)\neq\left\{ e\right\} $
for $r\geq3$, i.e. we do not know if the congruence subgroup property
holds for $\Phi_{2,r}$ for $r\geq3$ or not. Note that our proofs
of Theorems \ref{thm:.metabelian 2} and \ref{thm: metabelian 3}
claiming that $\Phi=\Phi_{2}=\Phi_{2,2}$ and $\Phi=\Phi_{3}$ do
not satisfy the CSP were based on two facts:
\begin{enumerate}
\item $Aut\left(\Phi\right)$ is large, and hence every finite group is
involved in $\widehat{Aut\left(\Phi\right)}$, and
\item not every finite group is involved in $Aut(\hat{\Phi})$.
\end{enumerate}
Now, for $\Phi=\Phi_{d,r}$, the free solvable group on $d\geq2$
generators and solvability length $r$, part 2 is valid for $1\leq r\leq2$
and every $d$ (with the same proof as for $d=3$ in $\varoint$\ref{sec:CSP for 3}).
But, as $C\left(\Phi_{d,1}\right)=\left\{ e\right\} $ for every $d\geq3$,
and $C\left(\Phi_{d,2}\right)$ is abelian for every $d\geq4$ (cf.
\cite{key-7}), part 1 is not valid in these cases. On the other hand,
for $\Phi=\Phi_{2,r}$ or $\Phi=\Phi_{3,r}$, part 1 is still true
for every $r\geq2$ but not part 2. Infact, we have:
\begin{prop}
Let $\Phi_{d,r}$ be the free solvable group on $d\geq2$ generators
and solvability length $r$. Then if $r\geq3$, then every finite
group $H$ is involved in $Aut(\hat{\Phi}_{d,r})$.\end{prop}
\begin{proof}
By the same arguments of (\cite{key-12}, 5.2), it can be deduced
from Gaschutz's Lemma that for every surjective homomorphism $\pi:\hat{\Phi}_{d,r}\to\Gamma$
where $\Gamma$ is finite, the homomorphism
\[
Aut(\hat{\Phi}_{d,r})\geq\left\{ \sigma\in Aut(\hat{\Phi}_{d,r})\,|\,\sigma\left(\ker\pi\right)=\ker\pi\right\} \to Aut\left(\Gamma\right)
\]
is surjective. Thus, for proving our proposition it suffices to show
that $\hat{\Phi}_{d,r}$ has a finite quotient $\Gamma$ such that
$H$ is involved in $Aut\left(\Gamma\right)$. Now, by Cayley's Theorem,
$H$ is a subgroup of $Sym\left(n-1\right)$ for some $n$ and the
later is a subgroup of $SL_{n}\left(p\right)$ for every prime $p$.
Thus, the next lemma due to Robert Guralnick, finishes the proof of
the proposition.\end{proof}
\begin{lem}
For every $n\geq2$, there exists a prime $p$ and a finite group
$\Gamma$ generated by two elements and of solvability length three,
such that $SL_{n}\left(p\right)$ is involved in $Aut(\Gamma)$. \end{lem}
\begin{proof}
Fix a prime $r$ such that $r>n+1$. Using Dirichlet's Theorem, pick
a prime $p$ such that $r$ divides $p-1$. Consider now the general
affine group 
\[
\Delta=AGL_{1}\left(r\right)=\left\{ \left(\begin{array}{cc}
a & 0\\
b & 1
\end{array}\right)\,|\,a\in\mathbb{F}_{r}^{*},\,b\in\mathbb{F}_{r}\right\} =\mathbb{F}_{r}\rtimes\mathbb{F}_{r}^{*}.
\]
Then $\Delta$ is of order $r\left(r-1\right)$. In addition, as $r|\left(p-1\right)$,
$\mathbb{F}_{p}$ contains the unit roots of order $r$, fix one of
them $\xi\neq1$, and consider the diagonal matrix 
\[
D=\left(\begin{array}{ccc}
\xi & \cdots & 0\\
\vdots & \ddots & \vdots\\
0 & \cdots & \xi^{r-1}
\end{array}\right)\in GL_{r-1}\left(p\right).
\]
Now, we can embed $\Delta$ in $GL_{r-1}\left(p\right)$ by sending
an element $b\in\left\{ 0,\ldots,r-1\right\} =\mathbb{F}_{r}$ to
the diagonal matrix $D^{b}$ (giving rise to a subgroup $N=\left\{ D^{b}\,|\,b\in\mathbb{F}_{r}\right\} $)
and an element $a\in\mathbb{F}_{r}^{*}$ to the permutation matrix
which normalizes $N$, sending $D^{b}$ to $D^{ba}$. So $\Delta$
has a module $V$ of dimension $r-1$ over $\mathbb{F}_{p}$. Now,
every $\Delta$-submodule of $V$ is also $N$-submodule. The $N$-submodules
are direct sums of different one dimensional $N$-modules, the eigen-spaces
of $D^{1}$, on which $\mathbb{F}_{r}^{*}$ acts transitively. We
deduce that $V$ is an irreducible module.

Denote now $W=\oplus_{i=1}^{r-2}V$ and using the obvious action of
$\Delta$ on $W$, define: $\Gamma=W\rtimes\Delta$. We claim that
$\Gamma$ is generated by two elements. By the description above,
it is clear why $\Delta$ is generated by two element, one of them
is $D\in\mathbb{F}_{r}$ and we denote the other one by $S\in\mathbb{F}_{r}^{*}$.
Let us now define 
\[
D'=\left((\vec{e}_{1},\ldots,\vec{e}_{r-2}),D\right),\,S'=\left((\vec{0},\ldots,\vec{0}),S\right)\in W\rtimes\Delta
\]
where $\left\{ \vec{e}_{1},\ldots,\vec{e}_{r-1}\right\} $ is the
standard basis of $V$. For a $1\leq j\leq r-1$ denote $\eta=\xi^{j}$.
Note, that for every $1\leq k\leq r-2$, $1+\eta+\ldots+\eta^{k}=\frac{1-\eta^{k+1}}{1-\eta}\neq0$.
It follows that $D'^{k}=\left((\alpha_{1}\vec{e}_{1},\ldots,\alpha_{r-2}\vec{e}_{r-2}),D^{k}\right)$
where $0\neq\alpha_{i}\in\mathbb{F}_{p}$ for every $1\leq k\leq r-2$.
Now, there is a power $S^{l}$ of $S$, $1\leq l\leq r-2$, which
sends $\vec{e}_{r-1}$ to $\vec{e}_{1}$. We have also $S^{l}DS^{-l}=D^{r-k}$
for some $1\leq k\leq r-2$. Thus, for some $0\neq\alpha_{i}\in\mathbb{F}_{p}$,
we can write:
\begin{eqnarray*}
w & = & S'^{l}D'S'^{-l}D'^{k}\\
 & = & ((\vec{0},\ldots,\vec{0}),S^{l})((\vec{e}_{1},\ldots,\vec{e}_{r-2}),D)((\vec{0},\ldots,\vec{0}),S^{-l})((\alpha_{1}\vec{e}_{1},\ldots,\alpha_{r-2}\vec{e}_{r-2}),D^{k})\\
 & = & ((S^{l}(\vec{e}_{1}),\ldots,S^{l}(\vec{e}_{r-2})),S^{l}DS^{-l})((\alpha_{1}\vec{e}_{1},\ldots,\alpha_{r-2}\vec{e}_{r-2}),D^{k})\\
 & = & (S^{l}(\vec{e}_{1})+D^{r-k}(\alpha_{1}\vec{e}_{1}),\ldots,S^{l}(\vec{e}_{r-2})+D^{r-k}(\alpha_{r-2}\vec{e}_{r-2}),I)\in W.
\end{eqnarray*}

Now, as $S^{l}$ sends $\vec{e}_{r-1}$ to $\vec{e}_{1}$, $\vec{e}_{1}$
does not appear in any entry of $w$ except the first one. 

Observe now, that the diagonals of $D^{0},\ldots,D^{r-2}$, considered
as column vectors of $V=\mathbb{F}_{p}^{r-1}$, form a basis for $V$
as the matrix:
\[
\left(\begin{array}{cccc}
1 & \xi & \cdots & \xi^{r-2}\\
1 & \xi^{2} & \cdots & \xi^{2\left(r-2\right)}\\
\vdots & \vdots &  & \vdots\\
1 & \xi^{r-1} & \cdots & \xi^{\left(r-1\right)\left(r-2\right)}
\end{array}\right)
\]
is a Vandermonde matrix, and therefore invertible. Thus, there is
a linear combination 
\[
C=\beta_{0}D^{0}+\ldots+\beta_{r-2}D^{r-2}=\left(\begin{array}{cccc}
1 & 0 & \cdots & 0\\
0 & 0\\
\vdots &  & \ddots & \vdots\\
0 & 0 & \cdots & 0
\end{array}\right),\,\,\,\beta_{i}\in\mathbb{F}_{p}.
\]
Now, observe that $D'$ acts on $W$ by conjugation via the action
of $D$ on $V$. Thus, we obtain an action of $C$ on $W$ via its
action on $V$, in which $C\left(w\right)$ has $\vec{0}$ in every
entry except the first one. This shows, as $V$ is irreducible, that
the first copy of $V$ in $W$ is inside the group generated by $D'$
and $S'$. In a similar way, all the $r-2$ copies of $V$ are generated
by $D'$ and $S'$, so $\Gamma$ is generated by two elements.

Now, $\Delta\times SL_{r-2}\left(p\right)$ acts on $W=\oplus_{i=1}^{r-2}V=V\otimes\mathbb{F}_{p}^{r-2}$
in a obvious way. Thus $\Gamma=W\rtimes\Delta$ is normal in $W\rtimes\left(\Delta\times SL_{r-2}\left(p\right)\right)$,
so $SL_{r-2}\left(p\right)$ is involved in $Aut\left(\Gamma\right)$.
\end{proof}
Let us remark that while we do not know the answer to the congruence
subgroup problem for free solvable groups on two generators and solvability
rank $r$ (unless $r=1$ or $2$), the situation with free nilpotent
groups on two generators is easier:
\begin{prop}
For every free nilpotent group on two generators $\Gamma$, the congruence
kernel contains a copy of $\hat{F}_{\omega}$ - the free profinite
group on countable number of generators.\end{prop}
\begin{proof}
It is known that if $\hat{\Gamma}$ is a pro-nilpotent group, then
the kernel of the map $Aut(\hat{\Gamma})\to Aut(\hat{\Gamma}/\overline{\Gamma'})$
is pro-nilpotent (cf. \cite{key-12}, 5.3). Thus, if $\Gamma$ is
a free nilpotent group (of arbitrary class) then by the same arguments
we brought before, there exists a finite group which is not involved
in $Aut(\hat{\Gamma})$. On the other hand, if $\Gamma$ is free nilpotent
group on two generators, then $Aut\left(\Gamma\right)$ is large,
as it can be mapped onto $GL_{2}\left(\mathbb{Z}\right)$ \footnote{In general, the kernel of the map $Aut\left(\Gamma\right)\to GL_{2}\left(\mathbb{Z}\right)$
strictly contains $Inn\left(\Gamma\right)$ (cf. \cite{key-14-1},
\cite{key-15})}. Thus, $\hat{F}_{2}$ is a subgroup of $\widehat{Aut\left(\Gamma\right)}$
and $C\left(\Gamma\right)\cap\hat{F}_{2}$ is non-trivial, hence contains
a copy of $\hat{F}_{\omega}$ (cf. \cite{key-8}).
\end{proof}
Our last remark is about the CSP for subgroups of automorphism groups.
Considering the classical congruence subgroup problem, one can take
$G$ to be a subgroup of $GL_{n}\left(R\right)$ where $R$ is a commutative
ring, and ask whether every finite index subgroup of $G$ contains
a  subgroup of the form $\ker\left(G\to GL_{n}\left(R/I\right)\right)$
for some finite index ideal $I\vartriangleleft R$. This direction
of generalization of the classical CSP has been studied intensively
during the second half of the $20^{\underline{th}}$ century (cf.
\cite{key-15-1}, \cite{key-16}). One can ask for a parallel generalization
for automorphism groups or outer atomorphism groups. I.e. let $G\leq Aut\left(\Gamma\right)$
(resp. $G\leq Out\left(\Gamma\right)$), does every finite index subgroup
of $G$ contain a principal congruence subgroup of the form $\ker\left(G\to Aut\left(\Gamma/M\right)\right)$
(resp. $\ker\left(G\to Out\left(\Gamma/M\right)\right)$) for some
characteristic finite index subgroup $M\leq\Gamma$?

Now, let $\pi_{g,n}$ be the fundamental group of $S_{g,n}$, the
surface of genus $g$ with $n$ punctures, such that $\chi\left(S_{g,n}\right)=2-2g-n\leq0$.
Then, there is an injective map of $PMod\left(S_{g,n}\right)$, the
pure mapping class group, into $Out\left(\pi_{g,n}\right)$ (cf. \cite{key-20},
chapter 8). Thus, one can ask the CSP for $PMod\left(S_{g,n}\right)$
as a subgroup of $Out\left(\pi_{g,n}\right)$. Considering the above
problem, it is known that:
\begin{thm}
For $g=0,1,2$ and every $n>0$, $PMod\left(S_{g,n}\right)$ has the
CSP.
\end{thm}
The cases for $g=0$ were proved by \cite{key-16-1} and in \cite{key-18},
the cases for $g=1$ were proved by \cite{key-3}, and the cases for
$g=2$ where proved by \cite{key-19}. It can be shown that for every
$n>0$, $\pi_{g,n}\cong F_{2g+n-1}$ = the free group on $2g+n-1$
generators. Thus, the above cases give an affirmative answer for various
subgroups of the outer aoutomorphism group of finitely generated free
groups. Though, the CSP for the full $Out\left(F_{d}\right)$ where
$d\geq3$ is still unsettled.

Institute of Mathematics\\
The Hebrew University\\
Jerusalem, ISRAEL 91904\\
\\
davidel-chai.ben-ezra@mail.huji.ac.il\\
alex.lubotzky@mail.huji.ac.il
\end{document}